\documentclass[11pt, reqno]{amsart}

\usepackage{amsmath}
\usepackage{amssymb}
\usepackage{amsrefs}
\usepackage{amsthm}
\usepackage[english]{babel}
\usepackage[autostyle]{csquotes}
\usepackage{bbm}
\usepackage{caption}
\usepackage[autostyle]{csquotes}
\usepackage{enumerate}
\usepackage[shortlabels]{enumitem}
\usepackage{float}
\usepackage[margin=1.25in]{geometry}
\usepackage[colorlinks=true, urlcolor=blue, linkcolor=black]{hyperref}
\usepackage{tikz}
\usepackage{xcolor}
\usetikzlibrary{backgrounds}
\usepackage{quiver}
\usepackage{ stmaryrd }
  
\theoremstyle{plain}
\newtheorem{thm}{Theorem}[section]
\newtheorem{cor}[thm]{Corollary}

\newtheorem{lem}[thm]{Lemma}

\theoremstyle{definition}
\newtheorem{defn}[thm]{Definition}
\newtheorem{rmk}[thm]{Remark}
\newtheorem{eg}[thm]{Example}

\usetikzlibrary{cd}

\newcommand{\Bun}{\operatorname{Bun}}

\newcommand{\Center}{\operatorname{Z}}

\newcommand{\Coh}{\operatorname{Coh}}

\newcommand{\IndCoh}{\operatorname{IndCoh}}

\newcommand{\End}{\operatorname{End}}

\newcommand{\Hom}{\operatorname{Hom}}

\newcommand{\Loc}{\operatorname{Loc}}

\newcommand{\Mod}{\operatorname{Mod}}

\newcommand{\Perv}{\operatorname{Perv}}

\newcommand{\Tr}{\operatorname{Tr}}
\newcommand{\Vect}{\operatorname{Vect}}

\newcommand{\SO}{\operatorname{SO}}

\newcommand{\opname}[1]{\operatorname{#1}}

\newcommand{\aff}{\operatorname{aff}}

\newcommand{\fr}{\operatorname{fr}}
\newcommand{\id}{\operatorname{id}}

\newcommand{\op}{\operatorname{op}}

\newcommand{\pt}{\operatorname{pt}}

\newcommand{\sh}{\operatorname{sh}}

\newcommand{\bbA}{\mathbb{A}}
\newcommand{\bbC}{\mathbb{C}}
\newcommand{\bbD}{\mathbb{D}}

\newcommand{\bbG}{\mathbb{G}}

\newcommand{\calA}{\mathcal{A}}

\newcommand{\calE}{\mathcal{E}}
\newcommand{\calEnd}{\mathcal{E}\kern -.5pt nd}
\newcommand{\calF}{\mathcal{F}}
\newcommand{\calG}{\mathcal{G}}
\newcommand{\calH}{\mathcal{H}}
\newcommand{\calHom}{\mathcal{H}\kern -.5pt om}

\newcommand{\calK}{\mathcal{K}}
\newcommand{\calL}{\mathcal{L}}
\newcommand{\calM}{\mathcal{M}}
\newcommand{\calN}{\mathcal{N}}
\newcommand{\calO}{\mathcal{O}}

\newcommand{\calS}{\mathcal{S}}

\newcommand{\calZ}{\mathcal{Z}}

\newcommand{\frakg}{\mathfrak{g}}

\newcommand{\catname}[1]{\textbf{#1}}

\newcommand{\bs}{\backslash}
\newcommand{\isomfrom}{\xleftarrow{\sim}}
\newcommand{\isomto}{\xrightarrow{\sim}}

\newcommand{\comment}[1]{}

\newcommand{\xto}[1]{\xrightarrow{#1}}

\newcommand{\inclto}{\hookrightarrow}

\newcommand{\toto}{\rightrightarrows}

\newcommand{\otimess}{\overset{!}{\otimes}}

\def\al#1\eal{\begin{align*}#1\end{align*}}

\title{The Large Affine Hecke Category is Calabi-Yau}
\date{}
\author{Yuji Okitani}

\begin{document}

\begin{abstract}
We show that the large affine Hecke category defines an oriented, fully extended topological field theory. More generally, we establish conditions under which ind-coherent convolution categories define such theories, analogously to known results for finite Hecke categories and $D$-modules.
\end{abstract}

\maketitle

\section{Introduction}
Local geometric Langlands studies categorical loop group representations associated to a reductive group $G$. The tamely ramified part of this is controlled by the affine Hecke category, playing the same role as the affine Hecke algebra in $p$-adic representation theory.

A feature specific to the categorical setting is the choice to work with either small categories such as $\Coh$, or with large categories such as $\IndCoh$. The purpose of this paper is to highlight a difference between small and large affine Hecke categories. We show that the large version is 2-dualizable while the small version is not. We show moreover that the large version admits a Calabi-Yau structure.

The affine Hecke category admits several variants independently to the choice in size. In order to state our main theorem, we will fix one such variant.

\begin{defn}
    The affine Hecke category is the \textit{large} monoidal category,
    \[\calH^{\aff} := \sh_{\text{nilp}}(I^\circ\bs G((t))/I^\circ),\]
    where $I^\circ$ is the pro-unipotent radical of the Iwahori.
\end{defn}

\begin{thm}\label{thm:introAHC}
    The affine Hecke category $\calH^{\aff}$ is 2-dualizable with a natural Calabi-Yau structure, so that it defines a fully extended and oriented 2-dimensional TQFT. In particular, there is a natural identification of the center of $\calH^{\aff}$ with its trace.
\end{thm}

Again, the emphasis is that $\calH^{\aff}$ is a \textit{large} monoidal category; our results continue to hold for some other variants of large affine Hecke categories. See Remark~\ref{rmk:variant} for a discussion on this.

\subsection{Idea of the Proof}
We work with the Langlands dual description of the affine Hecke category, under the universal variant of Bezrukavnikov's equivalence due to Dhillon and Taylor \cite{dhillonTameLocalBetti2025}. This presents $\calH^{\aff}$ as a convolution category of ind-coherent sheaves. We show that under suitable conditions, all such convolution categories are $2$-dualizable and admit a natural Calabi-Yau structure.

The strategy of the proof is an adaptation of the work of Ben-Zvi and Nadler \cite{ben-zviCharacterTheoryComplex2015}, who prove the 2-dualizability and Calabi-Yau property for the finite Hecke category, by studying general $D$-module convolution categories.

\subsection{Small vs Large}
The difference between small and large affine Hecke categories is most clearly demonstrated at the level of traces and centers. The following table gives spectral descriptions of these two invariants, where $G^\vee$ denotes the Langlands dual group and $E:=S^1\times S^1$.

\begin{center}
\begin{tabular}{ c|c|c } 
 & Small & Large \\ 
\hline
 Trace & $\Coh_{\calN}(\Loc_{G^\vee}(E))$ & $\IndCoh_{\calN}(\Loc_{G^\vee}(E))$ \\ 
 Center & $\Coh_{\text{prop}/{\frac{G^\vee}{G^\vee}}}(\Loc_{G^\vee}(E))$ & $\IndCoh_{\calN}(\Loc_{G^\vee}(E))$ \\ 
\end{tabular}
\end{center}

\begin{itemize}
    \item The left column of the table is calculated by Ben-Zvi, Nadler and Preygel in \cite{ben-zviSpectralIncarnationAffine2017}. They are each categories of coherent sheaves on the same scheme, but with conditions of starkly different flavors. The trace is given by a certain singular support constraint, which controls the manner in which the coherent sheaf may fail to be perfect. The condition in the center imposes that the sheaf remains coherent after pushforward along a certain morphism.
    \item The top row of the table comes from the observation that taking trace commutes with taking ind-completion, each being colimit constructions.
    \item The right column of the table then follows from Theorem~\ref{thm:introAHC}.
\end{itemize}

We learn that the center depends on the choice of size and agrees with the trace only in the case of the large version.

\subsection{Affine Character Sheaves}
The trace of the affine Hecke category is also known as the category of affine character sheaves and has been a focal object in categorical representation theory and local geometric Langlands \cites{ben-zviCoherentSpringerTheory2024,ben-zviEllipticSpringerTheory2015,ben-zviSpectralIncarnationAffine2017,bouthierPerverseSheavesInfinitedimensional2022,gorskyTraceAffineHecke2022}.

The center of the affine Hecke category in the context of local geometric Langlands is less well studied. We argue that this is in part due to subtleties described above. If we are willing to work with the large affine Hecke category then the center will admit a theory as well-behaved as, and which should enhance that of affine character sheaves.

\begin{rmk}
    The center of the affine Hecke category has been studied in the work of Gaitsgory \cite{gaitsgoryConstructionCentralElements2000}, and this is applied in \cite{bezrukavnikovCategoricalApproachStable2018}. That said, for both the primary objective is to produce many central elements in the affine Hecke \textit{algebra}. See also Remark~\ref{rmk:gaitsgory}.
    
    Many other works study the center, as well as the trace of the affine Hecke algebra. For example \cites{cautisCenterGenericAffine2025,ciubotaruCocentersRepresentationsAffine2014,heCentersCocenters$0$Hecke2015}.
\end{rmk}

\subsection{Application to 3d Mirror Symmetry}
In a different but closely related direction, our general results on ind-coherent convolution categories have applications to the 3d mirror symmetry of Gammage, Hilburn and Mazel-Gee \cite{gammagePerverseSchobers3d2023}. Their mirror symmetry is formulated as an equivalence of 2-categories, where the $A$-side is modelled by the 2-category $2\Perv(\bbC,0)$ of perverse schobers on a disc. In parallel to the story in geometric Langlands, the $B$-side is modelled by the 2-category of modules for a certain convolution category of ind-coherent sheaves. In the same way as for the affine Hecke category, we deduce the following.

\begin{thm}\label{thm:intro3d}
    The 2-category $2\Perv(\bbC,0)$ is 2-dualizable with a natural Calabi-Yau structure, so that it defines a fully extended and oriented 2-dimensional TQFT. In particular, there is a natural identification of the center of $2\Perv(\bbC,0)$ with its trace.
\end{thm}

The trace of this category is calculated in \cite{gammageBettiTatesThesis2025}. As before, our results indicate that the center is also worth studying.

\subsection{Organization of the Paper}
In Section~\ref{section:tft}, we recall some definitions in topological field theory as developed in \cite{lurieClassificationTopologicalField2009}. We will specialize to a particular choice of symmetric monoidal 2-category and recall the dualizability criterion for monoidal categories from \cite{ben-zviCharacterTheoryComplex2015}.

In Section~\ref{section:indcoh}, we outline the algebro-geometric setting we will work in. We then review properties of ind-coherent sheaves. These will follow \cites{gaitsgoryIndcoherentSheaves2012,drinfeldFinitenessQuestionsAlgebraic2012}.

In Section~\ref{section:conv} we prove the main technical results in the general setting of ind-coherent convolution categories.

In Section~\ref{section:applications} we apply these technical results to the affine Hecke category and prove Theorem~\ref{thm:introAHC}. We also outline our application to 3d mirror symmetry and prove Theorem~\ref{thm:intro3d}.

\subsection{Acknowledgments}
I am grateful to my PhD advisor David Nadler for many helpful discussions and for his continued support and open encouragement. This work was partially supported by the Simons Dissertation Fellowship.

\section{Topological Quantum Field Theories}\label{section:tft}
In this section, we provide a review of the language of topological quantum field theories, following \cite{lurieClassificationTopologicalField2009}. We then restrict our attention to TQFTs defined by monoidal categories and recall the key theorem of \cite{ben-zviCharacterTheoryComplex2015} needed to prove the dualizability of the affine Hecke category.

We will implicitly work with $(\infty,1)$-categories and $(\infty,2)$-categories.

\subsection{Review of TQFTs}
Topological quantum field theories are devices which encode structures, for example associated to a category, by interpreting these structures as invariants assigned to manifolds.

Given a symmetric monoidal 2-category $\catname{S}$ a fully extended framed 2-dimensional TQFT (abbreviated to framed 2TQFT or just 2TQFT) valued in $\catname{S}$ is a symmetric monoidal functor,
\[\calZ:\catname{Bord}^{\fr}_2\to\catname{S},\]
where $\catname{Bord}^{\fr}_2$ is the symmetric monoidal bordism 2-category of framed manifolds. Lurie's cobordism hypothesis states that the choice of $\calZ$ is equivalent to the choice of a 2-dualizable object in $\catname{S}$. This object is just $\calZ(\pt)$ where $\pt$ is the point with the standard framing.

\subsection{Traces and Centers}
Let us consider a 2TQFT $\calZ$, classified by a 2-dualizable object $A = \calZ(\pt)$. An important example of structure encoded by $\calZ$ is the center and trace of $A$, which we will now define. Denote by $S^1_{\Tr}$ the $S^1$ with the cylinder framing- this is the unique framing for which $S^1_{\Tr}$ admits an $\SO(2)$-action by rotation. Denote by $S^1_{\Center}$ the $S^1$ equipped with the annulus framing- this is the unique framing which admits the cup bordism from the empty framed 1-manifold. The trace of $A$ is defined to be
\[\Tr(A) := \calZ(S^1_{\Tr}),\]
and the center of $A$ is defined to be
\[\Center(A) := \calZ(S^1_{\Center}).\]
These are each objects inside the symmetric monoidal category $\End_{\catname{S}}(1_{\catname{S}})=:\Omega\catname{S}$. The 2TQFT encodes further structures. For example,
\begin{itemize}
    \item $\Tr(A)$ admits an $\SO(2)$-action,
    \item $\Center(A)$ admits an $\calE_2$-monoidal structure from the pair of pants construction,
    \item $\Tr(A)$ admits a structure of an $\calE_2$-module over $\Center(A)$ from a different pair of pants construction.
\end{itemize}

\subsection{Oriented TQFTS}
Often, there is a natural way in which to descend invariants from framed manifolds to oriented manifolds, along the forgetful functor $\catname{Bord}^{\opname{fr}}_2\to\catname{Bord}^{\opname{or}}_2$. An oriented 2TQFT valued in $\catname{S}$ is a symmetric monoidal functor,
\[\calZ:\catname{Bord}^{\opname{or}}_2\to\catname{S},\]
where $\catname{Bord}^{\opname{or}}_2$ is the symmetric monoidal bordism 2-category of oriented manifolds.

In \cite{lurieClassificationTopologicalField2009}*{Remark 4.2.7} it is shown that the choice of an oriented 2TQFT is equivalent to the choice of
\begin{itemize}
    \item A 2-dualizable object $A$ defining the underlying framed 2TQFT, and
    \item an $\SO(2)$-equivariant and non-degenerate trace functional $\tau:\Tr(A)\to 1_{\Omega\catname{S}}$ in $\Omega\catname{S}$. We refer to this as a Calabi-Yau structure on $A$.
\end{itemize}

In $\catname{Bord}^{\opname{or}}_2$, the framed circles $S^1_{\Center}$ and $S^1_{\Tr}$ become identified, and so we get an identification,
\[\Center(A) \isomto \Tr(A).\]

\subsection{Weakly Oriented TQFTS}
Let $A$ be a 2-dualizable object. The first obstruction for $A$ to admit a Calabi-Yau structure is given by the Serre automorphism $S_A:A\to A$. This is the 1-morphism assigned to the interval along which the framing makes one rotation. A weak Calabi-Yau structure is the choice of a trivialization $\id_A\isomto S_A$ in the space of automorphisms of $A$. As the name suggests this is a weaker structure, but is easier to work with and is in fact enough to write down an identification between the center and trace as above.

In parallel to Calabi-Yau structures, a weak Calabi-Yau structure is equivalent to the choice of a non-degenerate trace functional $\tau:\Tr(A)\to 1_{\Omega\catname{S}}$, with no $\SO(2)$-equivariance.

\subsection{Specializing to the Morita 2-Category}
Let $k$ be a field of characteristic zero. Denote by $\catname{St}_{k}$ the symmetric monoidal $1$-category of stable presentable $k$-linear $1$-categories, with morphisms given by colimit preserving functors. Consider the Morita $2$-category $\catname{S}:=\catname{Alg}(\catname{St}_{k})$. Recall that $\catname{S}$ has objects given by monoidal objects in $\catname{St}_{k}$, and morphisms given by the category of bimodules. This has a symmetric monoidal structure from tensor product of $k$-linear monoidal categories, with unit object $\Vect_{k}$. We study 2TQFTs valued in $\catname{S}$.

\begin{rmk}
    This will assign to closed 2-manifolds vector spaces and not numerical invariants. See also Remark~\ref{rmk:approx}.
\end{rmk}

For an object $\calA\in \catname{S}$, the trace and center are given by
\begin{align*}
    \Tr(\calA)&=\calA\otimes_{\calA\otimes\calA^{\op}}\calA,\\
    Z(\calA)&=\Hom_{\calA\otimes\calA^{\op}}(\calA,\calA).
\end{align*}

Let us now recall some properties of monoidal categories which are not Morita invariant. In other words, these properties are not intrinsic to the Morita category $\catname{S} = \catname{Alg}(\catname{St}_{k})$. One can think of these as depending on the choice of a 1-morphism, $\calA\to\Vect_{k}$, given by $\calA$ viewed as a bimodule between $\calA$ and $\Vect_{k}$. We first have an extrinsic finiteness notion.

\begin{defn}
    A monoidal category $\calA$ in $\catname{Alg}(\catname{St}_{k})$ is said to be semi-rigid it is compactly generated and every compact object $a\in\calA$ admits left and right monoidal duals $a^L,a^R$. We say that $\calA$ is rigid if moreover the monoidal unit $1_\calA\in\calA$ is compact.
\end{defn}

We may extend the assignments $a\mapsto a^L,a^R$ into monoidal equivalences $-^L,-^R:\calA\to\calA^{\op}$. We also have an extrinsic analogue of the Calabi-Yau notion.

\begin{defn}
    A pivotal structure on a semi-rigid monoidal category $\calA$ in $\catname{Alg}(\catname{St}_{k})$ is a monoidal identification, of the monoidal autoequivalence $-^{LL}$ with the identity of $\calA$.
\end{defn}

The following result from \cite{ben-zviCharacterTheoryComplex2015} is the key theorem connecting the extrinsic and intrinsic finiteness and Calabi-Yau notions.

\begin{thm}[{\cite{ben-zviCharacterTheoryComplex2015}*{Theorem 3.18}}]
    A semi-rigid monoidal category $\calA$ is 2-dualizable, and a pivotal structure on $\calA$ induces a weak Calabi-Yau structure on $\calA$.
\end{thm}

\section{Ind-coherent Sheaves}\label{section:indcoh}
In this section, we will provide a short summary of some properties of ind-coherent sheaves on stacks following \cites{gaitsgoryIndcoherentSheaves2012,drinfeldFinitenessQuestionsAlgebraic2012}. We will list the properties of $\IndCoh$ which we use in Section~\ref{section:conv}.

\subsection{Assumptions on Stacks}
We work with derived schemes and stacks over a field $k$ of characteristic zero. We assume all stacks are algebraic stacks mainly for the sake of simplicity. We also assume our algebraic stacks are (locally) almost of finite type and QCA, technical conditions that ensure we have a theory of ind-coherent sheaves with good functoriality and dualizability properties.

\begin{eg}
    Let $U$ be a scheme, almost of finite type with an action by $H$, a finite type affine group scheme. Then the global quotient $U/H$ is a QCA algebraic stack, almost of finite type. The relevant stacks in our two applications will be of this form.
\end{eg}

To a scheme $X$ almost of finite type, we associate the category $\IndCoh(X)$ defined to be the ind-completion of $\Coh(X)$. We may extend this and define $\IndCoh(X)$ for $X$ almost finite type, QCA algebraic stack via Kan extension.

\subsection{Properties}

We list below properties of $\IndCoh(X)$. We recall here that a morphism of stacks $p:X\to Y$ is proper if it is schematic and the pullback of $p$ to any scheme $S$ mapping to $Y$ is proper.

\begin{itemize}
    \item $\IndCoh(X)$ admits a symmetric monoidal product $\otimess$.
    \item For any morphism $f:X\to Y$, we have a continuous functors
        \begin{align*}
            f^!:\IndCoh(Y)\to\IndCoh(X),\\
            f_*:\IndCoh(X)\to\IndCoh(Y).
        \end{align*}
    \item $f^!$ is symmetric monoidal with respect to $\otimess$.
    \item For any Cartesian square,
        \[\begin{tikzcd}
            X' & X \\
            Y' & Y
            \arrow["u", from=1-1, to=1-2]
            \arrow["f'", from=1-1, to=2-1]
            \arrow["f", from=1-2, to=2-2]
            \arrow["v", from=2-1, to=2-2]
            \arrow[from=1-1, to=2-2, phantom, "\square"]
        \end{tikzcd}\]
        there is a base-change isomorphism $v^!f_*\simeq f'_*u^!$.
    \item For $f:X\to Y$, we have the projection formula $f_*(-\otimess (f^!-))\simeq (f_*-)\otimess-$.
    \item For a proper morphism $p:X\to Y$, $p_*$ is quasi-proper, that is, sends compact objects to compact objects. Moreover, $p_*$ is left adjoint to $p^!$. Recall here, proper morphisms between stacks are schematic by definition.
    \item For a morphism $q:X\to Y$ which is eventually coconnective, there is an additional pullback functor
        \[q^*:\IndCoh(Y)\to\IndCoh(X),\]
    which is continuous, quasi-proper and left adjoint to $q_*$. If furthermore $q$ is quasi-smooth, then $\alpha_q:= q^*\omega_Y$ is $\otimess$-invertible and the natural transformation
    \[q^*-\to \alpha_q\otimess q^!-\]
    is an equivalence.
    \item For any Cartesian square,
        \[\begin{tikzcd}
            X' & X \\
            Y' & Y
            \arrow["u", from=1-1, to=1-2]
            \arrow["q'", from=1-1, to=2-1]
            \arrow["q", from=1-2, to=2-2]
            \arrow["v", from=2-1, to=2-2]
            \arrow[from=1-1, to=2-2, phantom, "\square"]
        \end{tikzcd}\]
        where $q$ is quasi-smooth, there is a base-change isomorphism $u^!\alpha_q\simeq\alpha_{q'}$.
\end{itemize}

\begin{rmk}
    While the last two points may appear unfamiliar, it is natural from the point of view of six functor formalisms. Using terminology of \cite{scholzeSixFunctorFormalisms2025}- the assertion is that if $q$ is quasi-smooth then it is cohomologically smooth. We can compare to more familiar notation and observe that $\alpha_q\simeq \bbD\omega_q$. We also note that if $Y=\pt$ then $\alpha_q \simeq \calO_X$.
\end{rmk}

We have in addition the following properties.
\begin{itemize}[resume]
    \item The category $\IndCoh(X)$ is compactly generated. In particular, there is an identification $\IndCoh(X)^{\op}\isomto\IndCoh(X)^\vee$, given by $\calF\mapsto\Hom(\calF,-)$ for $\calF\in\Coh(X)$.
    \item The functor $-\boxtimes-:\IndCoh(X)\otimes\IndCoh(Y)\isomto\IndCoh(X\times Y)$ is an equivalence, where
        \[\calF\boxtimes\calG := p_X^!\calF\otimess p_Y^!\calG.\]
    \item There is a self-duality $\IndCoh(X)\isomto\IndCoh(X)^\vee$, given by $\calF\mapsto\Gamma(\calF\otimess-)$.
    \item Comparing the two descriptions of $\IndCoh(X)^\vee$ gives us the Serre duality $\bbD_X:\IndCoh(X)\to\IndCoh(X)^{\op}$, which is an involutive autoequivalence and satisfies 
        \[\Hom(\bbD_X\calF,-)\simeq\Gamma(\calF\otimess-)\]
    for $\calF\in\Coh(X)$.
    \item For $p:X\to Y$ proper, $\bbD_Yp_*\bbD_X\simeq p_*$, and for $q:X\to Y$ eventually coconnective, $\bbD_Xq^!\bbD_Y\simeq q^*$.
\end{itemize}
These properties and further compatibilities are encoded in a six functor formalism.

\section{Convolution Categories}\label{section:conv}
In this section, we study ind-coherent convolution categories and prove the main technical results needed for Theorem~\ref{thm:introAHC}.

\subsection{Setting}
Let $p:X\to Y$ be a proper morphism between smooth QCA algebraic stacks almost of finite type over a field $k$ of characteristic zero. We define a monoidal structure under convolution on $\calH := \IndCoh(X\times_YX)$ by the usual formulae,
\begin{align*}
    1_{\calH} &:= \delta_*\omega_X,\\
    \calK\ast\calL &:= p_{13*}(p_{12}\times p_{23})^!(\calK\boxtimes\calL)
\end{align*}
where $\delta:X\to X\times_YX$ is the diagonal and we have maps as follows.
\[\begin{tikzcd}
	&& {X\times_YX\times_YX} \\
	{X\times_YX} && {X\times_YX} && {X\times_YX}
	\arrow["{p_{12}}"', from=1-3, to=2-1]
	\arrow["{p_{13}}"', from=1-3, to=2-3]
	\arrow["{p_{23}}", from=1-3, to=2-5]
\end{tikzcd}\]
Let us also write $p_1,p_2:X\times_YX\toto X$ for the two projections. We remark that to define the full monoidal structure with all higher associativity data, we should define $X\times_YX$ as a monoidal object in a category of correspondences of derived stacks and pass to ind-coherent sheaf categories.

\begin{lem}\label{lem:qproper}
    The monoidal product is quasi-proper.
\end{lem}

\begin{proof}
    Observe that $p_{13}$ is a base-change of $p$ which is proper, and $p_{12}\times p_{23}$ is a base-change of $\Delta$ which is quasi-smooth. So both $p_{13*}$ and $(p_{12}\times p_{23})^!$ are quasi-proper.
\end{proof}

\subsection{Monoidal Duals}
We prove the main technical result, which refines \cite{ben-zviCoherentSpringerTheory2024}*{Theorem 3.2.5} by providing a formula for the monoidal duals.
\begin{thm}\label{thm:main}
    The monoidal category $\calH$ is rigid, and hence 2-dualizable. For a compact object $\calK\in\calH$, its monoidal duals are given by
    \begin{align*}
    \calK^R &:= (\bbD 1_{\calH})^{-1}\ast \bbD\sigma\calK,\\
    \calK^L &:= \bbD\sigma\calK\ast (\bbD 1_{\calH})^{-1}.
    \end{align*}
    Here $\sigma:\calH\to\calH^{\op}$ denotes the involution induced by the swap map, and $(\bbD 1_{\calH})^{-1}$ is the convolution-inverse of $\bbD 1_{\calH}$.
\end{thm}

The key calculation in the proof of Theorem~\ref{thm:main} is the following lemma, which describes how the monoidal structure of $\calH$ interacts with Serre duality $\bbD_{X\times_YX}$.

\begin{lem}\label{lem:duality}
    Given $\calK,\calL\in\calH$, we have natural identifications,
    \[\bbD(\calK\ast\calL)\simeq \bbD\calK\ast (\bbD1_{\calH})^{-1} \ast \bbD\calL.\]
\end{lem}

\begin{proof}
    Expanding the left hand side,
    \begin{align*}
        \bbD(\calK\ast\calL) &\simeq p_{13*}(p_{12}\times p_{23})^*\bbD(\calK\boxtimes\calL)\\
        &\simeq p_{13*}(\alpha_{p_{12}\times p_{23}}\otimess (p_{12}\times p_{23})^!(\bbD\calK\boxtimes\bbD\calL))\\
        &\simeq p_{13*}(q_2^!\alpha_{\Delta}\otimess (p_{12}\times p_{23})^!(\bbD\calK\boxtimes\bbD\calL)).
    \end{align*}
    We refer to Section~\ref{section:indcoh} for the definition of the sheaves denoted $\alpha_{-}$. In the last line, we have $\alpha_{p_{12}\times p_{23}}\simeq q_2^!\alpha_{\Delta}$ by base-change along the following Cartesian square whose vertical maps are quasi-smooth.
    \[\begin{tikzcd}
        {X\times_YX\times_YX} & X \\
        {(X\times_YX)\times (X\times_YX)} & {X\times X}
        \arrow["{q_2}", from=1-1, to=1-2]
        \arrow["{p_{12}\times p_{23}}"', from=1-1, to=2-1]
        \arrow["\Delta", from=1-2, to=2-2]
        \arrow[from=2-1, to=2-2]
        \arrow[from=1-1, to=2-2, phantom, "\square"]
    \end{tikzcd}\]

    Now consider the following diagram where the square is Cartesian. The correspondence on the right is the one which computes the convolution product of three objects.

    \[\begin{tikzcd}
        {X\times_YX\times_YX} & {X\times_YX\times_YX\times_YX} & {X\times_YX} \\
        {(X\times_YX)\times X\times (X\times_YX)} & {(X\times_YX)\times (X\times_YX)\times (X\times_YX)}
        \arrow[from=1-1, to=1-2]
        \arrow["{p_{12}\times q_2\times p_{23}}"', from=1-1, to=2-1]
        \arrow["{p_{14}}", from=1-2, to=1-3]
        \arrow[from=1-2, to=2-2]
        \arrow["{\id\times\delta\times\id}", from=2-1, to=2-2]
        \arrow[from=1-1, to=2-2, phantom, "\square"]
    \end{tikzcd}\]

    Our expression for $\bbD(\calK\ast\calL)$ is obtained by push-pull around the top of the diagram starting with the object $\bbD\calK\boxtimes\alpha_{\Delta}\boxtimes\bbD\calL$ in the bottom left. Going around the bottom by base-change,
    \[\bbD(\calK\ast\calL)\simeq \bbD\calK\ast(\delta_*\alpha_\Delta)\ast\bbD\calL.\]
    Finally, we have the identification
    \[\delta_*\alpha_\Delta\simeq (\bbD1_{\calH})^{-1},\]
    by setting $\calK=1_{\calH},\calL=\bbD1_{\calH}$.
\end{proof}

\begin{proof}[Proof of Theorem~\ref{thm:main}]
    The unit of the monoidal structure $\delta_*\omega_X$ is compact because the diagonal map $\delta:X\to X\times_YX$ of a schematic morphism is a closed embedding and $\omega_X$ is compact.

    We will exhibit the formula for the left dual. The right dual follows by conjugating by $\sigma$. Given $\calL,\calM\in\calH$, we provide functorial identifications,
    \[\Hom_\calH(\calL,\calK\ast\calM)\isomto\Hom_\calH(\calK^L\ast\calL,\calM).\]
    We may assume $\calL$ is compact as these identifications are colimit-preserving in this variable.
    
    We expand the left hand side using Serre duality and ``rotate''.
    \begin{align*}
        \Hom_\calH(\calL,\calK\ast\calM) &\simeq \Gamma\left(\bbD\calL\otimess p_{13*}(p_{12}\times p_{23})^!(\calK\boxtimes\calM)\right)\\
        &\simeq\Gamma\left((p_{13}^!\bbD\calL)\otimess (p_{12}^!\calK)\otimess (p_{23}^!\calM)\right)\\
        &\simeq\Gamma\left((p_{21}^!\sigma\calK)\otimess(p_{13}^!\bbD\calL)\otimess (p_{23}^!\calM)\right)\\
        &\simeq\Gamma\left(p_{23*}(p_{21}\times p_{13})^!(\sigma\calK\boxtimes\bbD\calL)\otimess \calM\right)\\
        &\simeq\Hom_\calH(\bbD(\sigma\calK\ast\bbD\calL),\calM).
    \end{align*}
    Serre duality applies in the last line as $\sigma\calK\ast\bbD\calL$ is compact by Lemma~\ref{lem:qproper}.
    
    Applying Lemma~\ref{lem:duality},
    \[\bbD(\sigma\calK\ast\bbD\calL) \simeq \bbD\sigma\calK\ast(\bbD1_{\calH})^{-1}\ast \calL.\]
    So we have the desired identification,
    \[\Hom_\calH(\calL,\calK\ast\calM)\isomto\Hom_\calH((\bbD\sigma\calK\ast(\bbD1_{\calH})^{-1})\ast\calL,\calM).\]
\end{proof}

\begin{cor}\label{cor:pivotal}
    A pivotal structure on $\calH$ is equivalent to a lift of $\bbD 1_{\calH}$ to a central element of $\calH$.
\end{cor}

\begin{proof}
    We compute $\calK^{LL}$ using Lemma~\ref{lem:duality},
    \[\calK^{LL} \simeq \bbD1_{\calH}\ast\calK\ast(\bbD1_{\calH})^{-1}.\]
    A trivialization of the monoidal autoequivalence $\calK\mapsto\calK^{LL}$ is precisely the structure lifting $\bbD1_{\calH}$ to a central element.
\end{proof}

\begin{eg}
    A natural source of central elements of $\calH$ comes from the monoidal functor $\delta_*p^!:\IndCoh(Y)\to\calH$. These central objects are analogous to scalar matrices in matrix algebras.

    In particular, a choice of $\calE\in\IndCoh(Y)$ together with an identification $p^!\calE\simeq\calO_X$ provides a central lift of $\bbD1_{\calH}\simeq\delta_*\calO_X$.
\end{eg}

\subsection{Calabi-Yau Structures}
The monoidal category $\calH$ in fact has a Calabi-Yau structure of a geometric origin. This is the analog of \cite{ben-zviCharacterTheoryComplex2015}*{Theorem 6.3}.

\begin{thm}\label{thm:CY}
    For any $p:X\to Y$ as in Theorem~\ref{thm:main}, the monoidal category $\calH$ admits a natural Calabi-Yau structure.
\end{thm}

\begin{proof}
    We construct our trace functional as a composition,
    \[\Tr(\calH)\isomto\IndCoh_{\Lambda}(LY)\inclto \IndCoh(LY)\xto{\Gamma(i^!-)}\Vect.\]

    The first map is the identification of \cite{ben-zviSpectralIncarnationAffine2017}*{Theorem 3.3.1}, where $\Lambda$ is a certain singular support condition cutting out a subcategory of $\IndCoh(LY)$. The third map is given by $\Gamma(i^!-)$ where $i:Y\to LY$ is the embedding of constant loops. All three maps are naturally $S^1$-equivariant so the composition is an $S^1$-equivariant trace functional.
    
    The non-degeneracy comes from observing that the composition
    \[\calH\otimes\calH \to \calH \to \Tr(\calH)\to \Vect\]
    is the self-pairing of Serre duality on $X\times_YX$.
\end{proof}

We notice that here we need no extra structure, in contrast to Corollary~\ref{cor:pivotal} where a choice of central lift $\bbD1_{\calH}$ was required. The discrepancy becomes apparent by comparing trace functionals.

Let us write $\tau^{\opname{CY}}$ and $\tau^{\opname{piv}}$ for the trace functionals $\Tr(\calH)\to\Vect$ constructed in Theorem~\ref{thm:CY} and in Corollary~\ref{cor:pivotal} respectively. On objects of $\calH$, we have
\begin{align*}
    \tau^{\opname{CY}}(\calK) &\simeq \Hom_{\calH}(\bbD1_{\calH},\calK),\\
    \tau^{\opname{piv}}(\calK) &\simeq \Hom_{\calH}(1_{\calH},\calK).
\end{align*}
So the weak Calabi-Yau structures arise from different trace functionals.

That said, we can relate these two structures as follows. Any two non-degenerate trace functionals differ by the insertion of an invertible central element $\calZ$. In our situation, this is the object for which there exists an identification
\[\tau^{\opname{piv}}(-)\simeq \tau^{\opname{CY}}(\calZ\ast-).\]
This object is nothing more than $\calZ \simeq \bbD1_{\calH}$ with the given central structure. We record a particular instance of this compatibility under which $\calZ$ is the monoidal unit.

\begin{cor}\label{cor:pivotalii}
    Suppose we are given an identification $\omega_X\simeq\calO_X$. This induces an identification $1_\calH\simeq\bbD1_{\calH}$ and the weak Calabi-Yau structure of Corollary~\ref{cor:pivotal} from the resulting central structure on $\bbD1_\calH$ is compatible with the Calabi-Yau structure of Theorem~\ref{thm:CY}.
\end{cor}

\section{Applications}\label{section:applications}
In this section, we prove the main result, Theorem~\ref{thm:introAHC}, on the affine Hecke category. We also outline our application to the 3d mirror symmetry of \cite{gammagePerverseSchobers3d2023} and prove Theorem~\ref{thm:intro3d}.

We will work over the field $k=\bbC$ on both the $A$-side and $B$-side.

\subsection{Local Geometric Langlands}
Let $G$ be a connected reductive group and $B\subseteq G$ a Borel. We recall again the definition of the large affine Hecke category, which is the variant we will work with. See Remark~\ref{rmk:variant} for a discussion on other variants.

\begin{defn}
    The affine Hecke category is the large monoidal category,
    \[\calH^{\aff} := \sh_{\text{nilp}}(I^\circ\bs G((t))/I^\circ),\]
    where $I^\circ$ is the pro-unipotent radical of the Iwahori.
\end{defn}

\begin{thm}\label{thm:AHC}
    The affine Hecke category $\calH^{\aff}$ is 2-dualizable and comes with a pivotal structure and a Calabi-Yau structure which are compatible.
\end{thm}

The key input is the universal variant of Bezrukavnikov's equivalence, which allows us to pass to ind-coherent convolution categories.

\begin{thm}[{\cite{dhillonTameLocalBetti2025}*{Theorem 1.2.5}}]\label{thm:DT}
    There is an equivalence of monoidal categories,
    \[\calH^{\aff}\simeq \IndCoh(X\times_YX),\]
    where $X=B^\vee/B^\vee, Y=G^\vee/G^\vee$, and $p:X\to Y$ is the map given by taking loops on $BB^\vee\to BG^\vee$.
\end{thm}

\begin{proof}[Proof of Theorem~\ref{thm:AHC}]
    We continue to use notation of Theorem~\ref{thm:DT}. The stacks $X,Y$ are smooth QCA stacks and $p$ is proper, so that $\calH^{\aff}$ is 2-dualizable by Theorem~\ref{thm:main}. We equip $\calH^{\aff}$ with a natural Calabi-Yau structure by Theorem~\ref{thm:CY}. Moreover, there is an isomorphism $\calO_X\simeq\omega_X$ by \cite{ben-zviCoherentSpringerTheory2024}*{Lemma 3.12}. By Corollary~\ref{cor:pivotalii}, $\calH^{\aff}$ has a natural pivotal structure which is compatible with its Calabi-Yau structure.
\end{proof}

\begin{cor}\label{cor:centerAHC}
    There is a natural identification of the center and trace of $\calH^{\aff}$,
    \[\Center(\calH^{\aff})\isomfrom\Tr(\calH^{\aff}).\]
    In particular, $\Center(\calH^{\aff})\simeq\IndCoh_{\calN}(\Loc_{G^\vee}(T^2))$, where $T^2=S^1\times S^1$, and where $\calN$ is the nilpotent cone.
\end{cor}

\begin{proof}
    The center and trace of $\calH^{\aff}$ are identified by the pivotal or Calabi-Yau structure. It is shown in \cite{ben-zviSpectralIncarnationAffine2017}*{Theorem 4.4.1} that the trace of $\calH^{\aff}$ is identified with $\IndCoh_{\calN}(\Loc_{G^\vee}(T^2))$.
\end{proof}

This torus $T^2$ decomposes into two copies of $S^1$, one coming from having taken the trace/center, and one coming from the loop in the loop group. We therefore expect that the $\calE_2$-monoidal structure on this category comes from pairs of pants on the first copy of $S^1$. Betti geometric Langlands predicts an automorphic description of the center, as the Betti automorphic category on an elliptic curve, $\sh_{\calN}(\Bun_G(E))$.

\begin{rmk}\label{rmk:gaitsgory}
    It is interesting to compare our calculation of the center to the construction of Gaitsgory's central functor which is a functor from the Satake category to $\calH^{\aff}$. It is shown in \cite{nadlerAutomorphicGluingFunctor2023} that the Satake category can be realized as the automorphic category on a genus zero curve. One might expect that this is related to our elliptic curve $E$ via a nodal degeneration.
\end{rmk}

\begin{rmk}\label{rmk:variant}
    We make a note of some other variants of affine Hecke category for which Theorem~\ref{thm:AHC} holds by presenting them as ind-coherent convolution categories.
    \begin{enumerate}
        \item If we take $X=\tilde{\calN}/G^{\vee}\to Y=\frakg^\vee/G^\vee$, where $\tilde{\calN}$ denotes the Springer resolution, we obtain the ind-completion of the affine Hecke category in the original equivalence of Bezrukavnikov \cite{bezrukavnikovTwoGeometricRealizations2016}.
        \item In the above, we may replace $G^\vee$ with $G^\vee\times\bbG_m$, where $\bbG_m$ acts by scaling. On the automorphic side this is expected to correspond to a mixed version of the affine Hecke category. This variant admits decategorifications to the classical affine Hecke algebra, see \cite{ben-zviCoherentSpringerTheory2024}*{Remark 1.1.4}.
    \end{enumerate}
\end{rmk}

\subsection{3d Mirror Symmetry}
Convolution categories naturally arise when considering 3d $B$-models, one model of which is Rozansky-Witten theory. Conjecturally this assigns to an algebraic symplectic variety, a 3-dimensional TQFT or equivalently a 3-dualizable 2-category of boundary conditions.

We can approximate these 2-categories as module categories for monoidal 1-categories by probing with testing objects. This allows us to decrease the categoricity, at the cost of not seeing the whole 2-category. For probes of a geometric origin, the resulting monoidal category is an ind-coherent convolution category.

The affine Hecke category arises in this way, by probing the stack $G^\vee/G^\vee$ by $B^\vee/B^\vee$. Recent work of Ben-Zvi, Nadler and Stefanich \cite{ben-zviPotentCategoricalRepresentations2025} proposes a precise definition of the full 2-category.

Another example arises from an instance of hypertoric 3d mirror symmetry due to Gammage, Hilburn and Mazel-Gee \cite{gammagePerverseSchobers3d2023}. The objects in the $A$-side 2-category are modelled by perverse schobers, which are a categorification of perverse sheaves. This is denoted below by $2\Perv$. Their mirror symmetry is an equivalence of 2-categories,
\[2\Perv(\bbC,0)\simeq \Mod(\calA).\]
Here, $\calA:=\IndCoh(X\times_YX)$ where $X=(\bbA^1\sqcup 0)/\bbG_m$, $Y=\bbA^1/\bbG_m$, and so our results apply.

\begin{rmk}\label{rmk:approx}
    Another cost of approximating our 3TQFT by monoidal categories is that this may decrease dualizability. This is why the TQFTs which arise from Theorem~\ref{thm:introAHC}, Theorem~\ref{thm:intro3d} are only 2-dimensional and assign vector spaces to manifolds of top dimension as opposed to numbers.
\end{rmk}

\begin{rmk}\label{rmk:boundaryconditions}
    Local geometric Langlands arises via an $S^1$-compactification of the maximally supersymmetric 4-dimensional gauge theory.

    Similarly, the hypertoric 3d mirror symmetry arises by compactifying the 4-dimensional $U(1)$-gauge theory along a certain interval decorated with boundary conditions. The $A$-side is the pairing of the Tate boundary condition with the Dirichlet boundary condition,
    \[(\bbA^1/\bbG_m)\times_{\pt/\bbG_m}(\bbG_m/\bbG_m)\simeq\bbC.\]
    On the $B$-side, we see dually, the pairing of the Tate boundary condition with the Neumann boundary condition,
    \[(\bbA^1/\bbG_m)\times_{\pt/\bbG_m}(\pt/\bbG_m)\simeq\bbA^1/\bbG_m.\]
    These pairs of boundary conditions are exchanged by relative Langlands duality \cite{ben-zviRelativeLanglandsDuality2024}.
\end{rmk}

\begin{thm}\label{thm:schob}
    The 2-category $2\Perv(\bbC,0)$ is 2-dualizable with a natural Calabi-Yau structure. In particular, there is a natural identification of the center of $2\Perv(\bbC,0)$ with its trace.
\end{thm}

\begin{proof}
    We pass to the ind-coherent convolution category $\calA$ via \cite{gammagePerverseSchobers3d2023}*{Theorem B}. The conditions of Theorem~\ref{thm:main} are easy to verify, giving us the 2-dualizability of $\calA$. Theorem~\ref{thm:CY} provides the Calabi-Yau structure.
\end{proof}

The main theorem of \cite{gammageBettiTatesThesis2025}, provides an $A$-side description of the trace of $\calA$,
\[\sh_\calS\left(\bbC((t))\right)\simeq \IndCoh(L(\bbA^1/\bbG_m)).\]
We deduce from Theorem~\ref{thm:schob} that these are also identified with the center. In particular, both sides admit a natural $\calE_2$-structure. We expect that the $\calE_2$-structure on $A$-side comes from a factorization structure on $\sh_\calS\left(\bbC((t))\right)$.

Going further, we are also at the liberty to evaluate our 2-dimensional theories on closed oriented 2-manifolds. In light of Remark~\ref{rmk:boundaryconditions}, and using language of \cite{ben-zviRelativeLanglandsDuality2024}, we expect the evaluations on these 2-manifolds to be the pairing of the Tate period sheaf with the Whittaker/Dirichlet period sheaf on the $A$-side, and of the Tate $L$-sheaf with the point/Neumann $L$-sheaf on the $B$-side.

\bibliography{ref2}

\end{document}